\documentclass{article}

\usepackage{amsmath}
\usepackage{amssymb}
\usepackage{amsthm}
\usepackage{graphicx}
\usepackage{xcolor}
\usepackage{bm}
\usepackage{booktabs}
\usepackage{caption}
\usepackage{enumitem}

\usepackage[ruled,linesnumbered]{algorithm2e}
\usepackage[a4paper,left=2.8cm,right=2.8cm,top=2.5cm,bottom=2.5cm]{geometry}
\usepackage{fancyhdr}

\usepackage{hyperref}
\definecolor{Mycolor}{RGB}{10,74,38}
\hypersetup{
  colorlinks=true,
  linkcolor=Mycolor,
  citecolor=Mycolor,
  urlcolor=Mycolor
}
\newtheorem{theorem}{Theorem}[section]
\newtheorem{corollary}[theorem]{Corollary}
\newtheorem{proposition}[theorem]{Proposition}
\newtheorem{lemma}[theorem]{Lemma}

\theoremstyle{definition}
\newtheorem{definition}[theorem]{Definition}

\theoremstyle{remark}

\numberwithin{equation}{section}

\newcommand{\Sd}{\mathbb S^d}

\newcommand{\dd}{\,\mathrm d}

\newcommand{\ran}{\operatorname{ran}}
\newcommand{\Proj}{\mathcal P}
\newcommand{\id}{\mathcal I}
\newcommand{\E}{\mathbb E}

\newcommand{\norm}[1]{\left\lVert #1\right\rVert}
\newcommand{\abs}[1]{\left\lvert #1\right\rvert}

\title{Positivity loss in bandlimited spectral reproduction on spheres}
\author{Hao-Ning Wu\thanks{School of Mathematical Sciences, Xiamen
University, Xiamen, Fujian, 361005, China (\texttt{hnwu@xmu.edu.cn})}}
\date{}

\begin{document}

\maketitle

\begin{abstract}
How small can the positivity loss be for an \(N\)-bandlimited spectral operator that exactly reproduces all modes up to degree \(L\)?  For spherical polynomial approximation on \(\mathbb S^d\), we prove that the smallest possible excess of the uniform operator norm above \(1\), equivalently the least positivity loss, is of sharp order $\left({L}/{(N+1)}\right)^2$ when \(1\le L< N\).  The lower bound follows from a Fej\'er peak test and a concentration estimate for bandlimited kernels, while a matching upper bound is obtained by correcting a positive Jackson operator with a smooth filter. We illustrate the result in three settings.  On the circle, taking \(N=sL-1\), this determines the sharp order of the generalized-projection constant above \(1\) and identifies the gap between the \(s^{-1}\) excess of delayed de la Vall\'ee--Poussin means and the optimal \(s^{-2}\) order. For filtered hyperinterpolation, whose operator norm has long been known to be uniformly bounded, we give a quantitative lower bound on its separation from the positivity threshold \(1\). Finally, we identify an operator-level obstruction to maximum principles.
\end{abstract}



\section{Introduction}

We consider spectral approximation on spheres \(\Sd\), where spherical harmonics provide the canonical frequency decomposition. A natural approximation operator should exactly reproduce the resolved low-frequency modes while keeping its output bandlimited. Exact reproduction preserves the information already resolved, whereas finite bandwidth keeps the approximation finite-dimensional. These two requirements, however, interact nontrivially with positivity. Orthogonal projections and other sharp spectral cutoffs typically have sign-changing kernels and may therefore map a nonnegative function to one that takes negative values. Allowing the output band to extend beyond the exactly reproduced modes creates a transition band in which oscillations can be reduced and stability improved. In this paper, we ask how much this additional spectral freedom can reduce the unavoidable loss of positivity.

Our aim is to determine the smallest possible positivity loss in terms of the two spectral
scales.  Let \(C(\Sd)\) be the real Banach space of continuous
functions on \(\Sd\), equipped with the uniform norm, and let
\(\mathbb P_n\) denote the space of real
spherical polynomials of degree at most \(n\).  Given \(L\le N\), we consider
operators $\mathcal{A}:\,C(\Sd)\rightarrow C(\Sd)$ satisfying the exact reproduction condition
\begin{equation}\label{eq:reproduction-constraint}
 \mathcal A p=p\quad \forall p\in\mathbb P_L,
\end{equation}
and the bandwidth condition
\begin{equation}\label{eq:bandwidth-constraint}
 \ran\mathcal A\subseteq\mathbb P_N.
\end{equation}
Thus \(L\) is the exact reproduction bandwidth, \(N\) is the output
bandwidth, and the degrees \(L+1,\ldots,N\) form the transition band between
the reproduced and discarded frequencies. Both constraints are essential.
Without \eqref{eq:bandwidth-constraint}, the identity operator $\id$ is admissible and
positive.  Without \eqref{eq:reproduction-constraint}, the projection onto
the constants,
\(
 \Proj_0f=\int_{\Sd}f\,\dd\sigma,
\)
is a finite-rank positive operator.  When the two constraints are imposed
together, positivity is impossible for \(L\ge1\). Indeed, a positive operator
reproducing \(\mathbb P_1\) must be the identity, which cannot satisfy
\eqref{eq:bandwidth-constraint} for any finite \(N\).

To measure the loss of positivity, we define the two-scale generalized-projection constant
\begin{equation}\label{eq:intro-lambda}
 \lambda_d(L,N)
 :=
 \inf\left\{
 \norm{\mathcal A}_{C(\Sd)\to C(\Sd)}:
 \mathcal A\in\mathcal L(C(\Sd)),
 \ \mathcal A|_{\mathbb P_L}=\id,
 \ \ran\mathcal A\subseteq\mathbb P_N
 \right\},
\end{equation}
where \(\mathcal L(C(\Sd))\) denotes the space of bounded linear operators on
\(C(\Sd)\).  The condition 
\[\mathcal A|_{\mathbb P_L}=\id\] means precisely that
\(\mathcal A\) satisfies the reproduction property
\eqref{eq:reproduction-constraint}.  When \(L=N\), the problem reduces to
the classical projection-constant problem for \(\mathbb P_N\); see
\cite{Grunbaum1960} for the general theory of projection constants and
\cite{DefantGalicerMansillaMastyloMuro2026} for spherical cases.  When \(L<N\), the additional degrees \(L+1,\ldots,N\) provide a
transition band that can be exploited by smooth spectral filters to obtain
degree-uniform stability bounds; see, e.g.,
\cite{Sloan2011,SloanWomersley2012}.
The extremal problem \eqref{eq:intro-lambda} asks how closely the optimal
operator norm can approach the positivity threshold \(1\), attained by
positive constant-preserving operators.

Since the constant function \(1\) belongs to \(\mathbb P_L\), every
operator $\mathcal A$ admissible in \eqref{eq:intro-lambda} is unital, that is,
\[\mathcal A1=1.\]
For a general bounded operator \(\mathcal A:C(\Sd)\to C(\Sd)\) with
\(\mathcal A1=1\), the Riesz--Markov theorem gives, for each \(x\in\Sd\), a
unique finite signed regular Borel measure \(\mu_x^{\mathcal A}\) representing
the bounded functional \(L_x(f):=\mathcal Af(x)\).  Thus
\begin{equation}\label{eq:Riesz-kernel}
 \mathcal Af(x)=\int_{\Sd}f(y)\,\dd\mu_x^{\mathcal A}(y),
 \qquad
 \norm{L_x}_{C(\Sd)^*}
 =\norm{\mu_x^{\mathcal A}}_{\mathrm{TV}}.
\end{equation}
Since \(\mathcal A1=1\), the measure has total mass one.  Let
\[\mu_x^{\mathcal A}
=\mu_{x,+}^{\mathcal A}-\mu_{x,-}^{\mathcal A}\]
be its Jordan decomposition.

\begin{definition}[Positivity loss]\label{def:positivity-loss}
The \emph{positivity loss} of a bounded operator
\(\mathcal A:C(\Sd)\to C(\Sd)\) satisfying \(\mathcal A1=1\) is defined by
\[
 \mathfrak p(\mathcal A)
 :=\sup_{x\in\Sd}\mu_{x,-}^{\mathcal A}(\Sd).
\]
\end{definition}

The following identity relates positivity loss to the operator norm.  In
particular, it shows that \(1\) is the smallest possible norm and that this
value is attained exactly by positive operators.
\begin{proposition}[Norm and positivity loss]
\label{prop:general-defect}
For every bounded operator \(\mathcal A:C(\Sd)\to C(\Sd)\) with
\(\mathcal A1=1\),
\begin{equation}\label{eq:general-norm-defect}
 \norm{\mathcal A}_{C(\Sd)\to C(\Sd)}
 =1+2\mathfrak p(\mathcal A).
\end{equation}
Moreover,
\begin{equation}\label{eq:general-undershoot}
 \mathfrak p(\mathcal A)
=\sup_{\substack{x\in\Sd,\,0\le f\le1}}
       \bigl(\mathcal Af(x)-1\bigr)=\sup_{\substack{x\in\Sd,\,0\le f\le1}}
       \bigl(-\mathcal Af(x)\bigr).
\end{equation}
In particular, \(\norm{\mathcal A}\ge1\), with equality if and only if
\(\mathcal A\) is positive.
\end{proposition}

\begin{proof}
By \eqref{eq:Riesz-kernel},
$\norm{\mathcal A}
 =\sup_{x\in\Sd}
  \norm{\mu_x^{\mathcal A}}_{\mathrm{TV}}$.
Since 
\[\mu_{x,+}^{\mathcal A}(\Sd)
 -\mu_{x,-}^{\mathcal A}(\Sd)
 =\mu_x^{\mathcal A}(\Sd)=1,\]
we have
\begin{equation*}
 \norm{\mu_x^{\mathcal A}}_{\mathrm{TV}}
 =
 \mu_{x,+}^{\mathcal A}(\Sd)
 +\mu_{x,-}^{\mathcal A}(\Sd)=
 1+2\mu_{x,-}^{\mathcal A}(\Sd).
\end{equation*}
Taking the supremum over \(x\) proves
\eqref{eq:general-norm-defect}.
For fixed \(x\in\Sd\) and \(0\le f\le1\),
\[
 -\int_{\Sd}f\,\dd\mu_x^{\mathcal A}
 =
 -\int_{\Sd}f\,\dd\mu_{x,+}^{\mathcal A}
 +\int_{\Sd}f\,\dd\mu_{x,-}^{\mathcal A}
 \le \mu_{x,-}^{\mathcal A}(\Sd).
\]
For the reverse inequality, mutual singularity gives a Borel set
\(E\subseteq\Sd\) such that
$\mu_{x,+}^{\mathcal A}(E)=0$ and
$\mu_{x,-}^{\mathcal A}(\Sd\setminus E)=0$.
By regularity, for every \(\varepsilon>0\) there are a compact set
\(F\subseteq E\) and an open set \(U\supseteq F\) such that
$\mu_{x,-}^{\mathcal A}(F)
 \ge \mu_{x,-}^{\mathcal A}(\Sd)-\varepsilon$ and
$\mu_{x,+}^{\mathcal A}(U)<\varepsilon$.
By Urysohn's lemma, there exists \(f\in C(\Sd)\) with
\(0\le f\le1\), \(f=1\) on \(F\), and \(f=0\) on
\(\Sd\setminus U\).  Therefore
\[
 -\int_{\Sd}f\,\dd\mu_x^{\mathcal A}
 \ge
 \mu_{x,-}^{\mathcal A}(F)
 -\mu_{x,+}^{\mathcal A}(U)
 \ge
 \mu_{x,-}^{\mathcal A}(\Sd)-2\varepsilon.
\]
Letting \(\varepsilon\rightarrow0\) shows that
\[
 \sup_{0\le f\le1}\bigl(-\mathcal Af(x)\bigr)
 =\mu_{x,-}^{\mathcal A}(\Sd).\]
Taking the supremum over \(x\) proves the second equality in
\eqref{eq:general-undershoot}.  Since $\mathcal Af(x)-1=-\mathcal A(1-f)(x)$,
and \(f\mapsto1-f\) preserves the class \(0\le f\le1\), we obtain the first equality in
\eqref{eq:general-undershoot}.  Finally,
\(\mathfrak p(\mathcal A)=0\) precisely when all the representing measures
are positive, i.e., when \(\mathcal A\) is positive.
\end{proof}

Applied to \eqref{eq:intro-lambda}, Proposition~\ref{prop:general-defect}
gives \(\lambda_d(L,N)\ge1\). Moreover, no admissible operator is positive when
\(L\ge1\) and \(N<\infty\).  Indeed, if \(\mathcal A\) were positive, then
each \(\mu_x^{\mathcal A}\) would be a probability measure.
For \(p_x(y):=x\cdot y\in\mathbb P_1\subseteq\mathbb P_L\), exact
reproduction gives
\[
 1=p_x(x)=\mathcal Ap_x(x)
   =\int_{\Sd}x\cdot y\,\dd\mu_x^{\mathcal A}(y).
\]
Since \(1-x\cdot y\ge0\), the measure \(\mu_x^{\mathcal A}\) must be
supported on \(\{x\}\), and hence equals the Dirac measure \(\delta_x\).
Thus \(\mathcal A=\id\), contradicting
\eqref{eq:bandwidth-constraint}.  Theorem~\ref{thm:intro} below strengthens
this rigidity by separating \(\lambda_d(L,N)\) quantitatively from \(1\).
The qualitative argument is standard in positive operator and Korovkin
theory; see, e.g., \cite{AltomareCampiti1994}. Our question is quantitative. In particular,
Proposition~\ref{prop:general-defect} also gives the quantitative identity
\[
 \lambda_d(L,N)-1
 =2\inf\left\{\mathfrak p(\mathcal A):\,{\mathcal A\in\mathcal L(C(\Sd)),\,
                   \mathcal A|_{\mathbb P_L}=\id,\,
                   \ran\mathcal A\subseteq\mathbb P_N}\right\}.
\]
Thus the excess norm in \eqref{eq:intro-lambda} is exactly twice the
smallest positivity loss compatible with the two spectral constraints.

We now quantify this unavoidable loss.  Our main result shows that the
smallest  positivity loss has order \((L/(N+1))^2\).  Thus a wider
transition band allows an exactly reproducing operator to approach
positivity, but only at a quadratic rate.
\begin{theorem}\label{thm:intro}
For every \(d\ge1\) and \(0<\kappa<1\), there are constants
\(0<c_d\le C_{d,\kappa}<\infty\) such that
\begin{equation}\label{eq:intro-main}
 c_d\left(\frac{L}{N+1}\right)^2
 \le \lambda_d(L,N)-1
 \le C_{d,\kappa}\left(\frac{L}{N+1}\right)^2,
\end{equation}
whenever \(1\le L\le\kappa N\). In terms of positivity loss, this reads
\begin{equation}\label{eq:intro-main-positivity-loss}
 \frac{c_d}{2}\left(\frac{L}{N+1}\right)^2
 \le
 \inf\left\{\mathfrak p(\mathcal A):\,{\mathcal A\in\mathcal L(C(\Sd)),\,
                   \mathcal A|_{\mathbb P_L}=\id,\,
                   \ran\mathcal A\subseteq\mathbb P_N}\right\}
 \le
 \frac{C_{d,\kappa}}{2}\left(\frac{L}{N+1}\right)^2.
\end{equation}
\end{theorem}

Theorem~\ref{thm:intro} connects with three related strands of approximation
theory and numerical analysis.  On the circle, the two-scale problem is
closely related to generalized projections and delayed de la Vall\'ee--Poussin
means.  A prescribed de la Vall\'ee--Poussin mean with output bandwidth
\(sL-1\) has excess norm of order \(s^{-1}\) and ceases to be minimal once
\(s>2\); see
\cite{DeregowskaFoucartLewandowskaSkrzypek2018,
DeregowskaLewandowska2015,Mehta2015}.  By contrast,
\eqref{eq:intro-main} shows that the optimal excess over the full admissible
class has the strictly smaller order \(s^{-2}\).  On higher-dimensional
spheres, filtered approximation and filtered hyperinterpolation use the
transition band to obtain uniform stability; see
\cite{Sloan2011,SloanWomersley2012,WangSloan2017}.  These results show that the operator
norm remains bounded independently of the degree, but do not determine how
close the best possible norm can come to the positivity threshold \(1\).
Theorem~\ref{thm:intro} quantifies this distance in terms of the ratio of the
reproduced bandwidth to the output bandwidth.  Finally,
\eqref{eq:intro-main-positivity-loss}, together with
Proposition~\ref{prop:general-defect}, shows that the least 
worst-case undershoot below \(0\), or overshoot above \(1\), has the same
quadratic order.

The paper is organized as follows.  Section~\ref{sec:setup} reviews the
spherical harmonic framework and identifies the positivity loss of a zonal
operator with the negative mass of its kernel.
Section~\ref{sec:lower} reduces the extremal problem to zonal operators and
proves the lower bound using a low-frequency peak polynomial and a
concentration estimate for bandlimited kernels.
Section~\ref{sec:upper} constructs a matching near-positive operator by
combining a positive Jackson operator with a bounded flat-top correction.
Section~\ref{sec:consequences} discusses three consequences, involving generalized projections on the circle,
filtered hyperinterpolation, and effective maximum principles for spectral
methods.

\section{Preliminaries}
\label{sec:setup}

Let \(\sigma\) be the surface measure on \(\Sd\), normalized by
\(\sigma(\Sd)=1\), and
let
\(
 \rho(x,y):=\arccos(x\cdot y)
\)
be the geodesic distance between $x$ and $y$.  For \(\ell\ge0\), let
\(\mathbb H_\ell:=\mathbb H_\ell(\Sd)\) be the space of real spherical
harmonics of degree \(\ell\) with dimension
\(Z(d,\ell):=\dim\mathbb H_\ell\).  The orthogonal harmonic decomposition is
\begin{equation}\label{eq:harmonic-decomposition}
 L^2(\Sd)=\bigoplus_{\ell=0}^{\infty}\mathbb H_\ell.
\end{equation}
The space of real spherical polynomials of degree at most \(n\) is
\[
 \mathbb P_n:=\bigoplus_{\ell=0}^{n}\mathbb H_\ell,
\]
where the dimension of $\mathbb{P}_n$ is 
\begin{equation}\label{eq:dim-Pi}
 D_{d,n}:=\dim\mathbb P_n
 =\binom{n+d}{d}+\binom{n+d-1}{d}\asymp_d (n+1)^d.
\end{equation}
Let \(\{Y_{\ell,j}\}_{j=1}^{Z(d,\ell)}\) be a real orthonormal basis of
\(\mathbb H_\ell\).  The reproducing kernel of $\mathbb{H}_\ell$ is
\[
 \mathcal Z_\ell(x,y)
 :=\sum_{j=1}^{Z(d,\ell)}Y_{\ell,j}(x)Y_{\ell,j}(y),
\]
and the orthogonal projection \(\Proj_\ell\) onto \(\mathbb H_\ell\) is then defined by
\begin{equation}\label{eq:harmonic-projection}
 \Proj_\ell f(x)
 =\int_{\Sd}\mathcal Z_\ell(x,y)f(y)\,\dd\sigma(y).
\end{equation}
By the addition formula,
\begin{equation}\label{eq:addition-formula}
 \mathcal Z_\ell(x,y)
 =Z(d,\ell)P_\ell^{(d)}(x\cdot y),
\end{equation}
where, for \(d\ge2\), \(P_\ell^{(d)}\) is the Gegenbauer polynomial of
degree \(\ell\) normalized by \(P_\ell^{(d)}(1)=1\).  For \(d=1\), we set
\(P_0^{(1)}=1\) and \(P_\ell^{(1)}(\cos\theta)=\cos(\ell\theta)\) for
\(\ell\ge1\).

Let \(\nu_d\) be the probability measure obtained by pushing \(\sigma\)
forward under \(y\mapsto x\cdot y\), given explicitly by
\[
 \dd\nu_d(t)
 =\frac{\Gamma((d+1)/2)}{\sqrt{\pi}\,\Gamma(d/2)}
  (1-t^2)^{d/2-1}\,\dd t.
\]
Let \(k\in L^1([-1,1],\nu_d)\) be real-valued and set
\(K(x,y):=k(x\cdot y)\).  By the Funk--Hecke
formula, the associated integral operator
\[
 \mathcal T_Kf(x)
 :=\int_{\Sd}K(x,y)f(y)\,\dd\sigma(y)
\]
is diagonal with respect to the spherical harmonic decomposition:
\begin{equation}\label{eq:multiplier}
\mathcal T_KY_{\ell,j}=\widehat k(\ell)Y_{\ell,j},
 \end{equation}
where 
\[\widehat k(\ell)
 :=\int_{-1}^{1}k(t)P_\ell^{(d)}(t)\,\dd\nu_d(t).\]
If \(k\) is an algebraic polynomial of degree at most \(N\), then \(K\) has
bandwidth at most \(N\), and Gegenbauer orthogonality gives the unique expansion
\begin{equation}\label{eq:zonal-kernel}
 K(x,y)
 =\sum_{\ell=0}^{N}a_\ell\mathcal Z_\ell(x,y),
 \qquad
 a_\ell=\widehat k(\ell).
\end{equation}
It follows from \eqref{eq:multiplier} that
\(\mathcal T_K|_{\mathbb P_L}=\id\) if and only if \(a_\ell=1\) for all
\(0\le\ell\le L\).  In particular, \(\mathcal T_K1=a_0\), so
\(\mathcal T_K\) is unital precisely when
\(a_0=1\), in which case
\[
 \int_{\Sd}K(x,y)\,\dd\sigma(y)
 =\mathcal T_K1(x)=1,
 \qquad x\in\Sd.
\]
We refer the reader to \cite{MR2934227,DaiXu2013} for these standard facts on spheres.

For such a unital zonal operator, the positivity loss of $\mathcal{T}_K$ is the negative mass of its kernel.
For a real-valued function \(u\), let
\(u_+:=\max\{u,0\}\) and \(u_-:=\max\{-u,0\}\).  The representing measure
\(\mu_x^{\mathcal T_K}\) is now explicit in the form of
\[
 \dd\mu_x^{\mathcal T_K}(y)=K(x,y)\,\dd\sigma(y).
\]
Its negative mass is therefore given by
\[
 \int_{\Sd}K_-(x,y)\,\dd\sigma(y).
\]
Let
\[
 O(d+1)
 :=\left\{g\in\mathbb R^{(d+1)\times(d+1)}:g^{\mathsf T}g=I\right\}
\]
be the orthogonal group of \(\mathbb R^{d+1}\).  This compact group acts
transitively on \(\Sd\) by \(x\mapsto gx\) and preserves both the inner
product and the measure \(\sigma\).
Since \(K(x,y)=k(x\cdot y)\) and \(\sigma\) is rotation invariant,
this integral is independent of \(x\).
For any \(x,x'\in\Sd\), we choose
\(R\in O(d+1)\) such that \(Rx=x'\).  The change of variables \(y=Rz\)
gives
\[
 \int_{\Sd}K_-(x',y)\,\dd\sigma(y)
 =\int_{\Sd}[k(Rx\cdot Rz)]_-\,\dd\sigma(z)
 =\int_{\Sd}[k(x\cdot z)]_-\,\dd\sigma(z).
\]
Moreover, since \(\mu_x^{\mathcal T_K}\) is absolutely continuous with
respect to \(\sigma\), its Jordan components are
\[
 \dd\mu_{x,+}^{\mathcal T_K}(y)=K_+(x,y)\,\dd\sigma(y),
 \qquad
 \dd\mu_{x,-}^{\mathcal T_K}(y)=K_-(x,y)\,\dd\sigma(y).
\]
Hence
\begin{equation}\label{eq:zonal-positivity-loss}
 \mathfrak p(\mathcal T_K)
 =\sup_{x\in\Sd}\mu_{x,-}^{\mathcal T_K}(\Sd)
 =\sup_{x\in\Sd}\int_{\Sd}K_-(x,y)\,\dd\sigma(y)
 =\int_{\Sd}K_-(x,y)\,\dd\sigma(y),
\end{equation}
where the last expression has the same value for every \(x\in\Sd\).
Proposition~\ref{prop:general-defect} now yields
\begin{equation}\label{eq:norm-negative-mass}
 \norm{\mathcal T_K}_{C(\Sd)\to C(\Sd)}
 =\sup_{x\in\Sd}\int_{\Sd}\abs{K(x,y)}\,\dd\sigma(y)
 =\int_{\Sd}\abs{K(x,y)}\,\dd\sigma(y)
 =1+2\mathfrak p(\mathcal T_K).
\end{equation}
Thus, for zonal operators, negative kernel mass is exactly the positivity
loss and half the excess norm.  By \eqref{eq:general-undershoot}, it is also
the worst violation of the interval \([0,1]\).

\section{Symmetry reduction and the lower bound}
\label{sec:lower}

To obtain the lower bound in Theorem~\ref{thm:intro}, we first reduce the extremal problem \eqref{eq:intro-lambda} from arbitrary admissible operators to
zonal operators.  Averaging an admissible operator over \(O(d+1)\) preserves reproduction and
bandwidth and does not increase its norm.  To state the resulting reduction,
we define
\begin{equation}\label{eq:zonal-admissible-class}
 \mathfrak T_d(L,N)
 :=\left\{\mathcal T_K:C(\Sd)\to\mathbb P_N:
 K\text{ is a real zonal kernel of bandwidth at most }N,\,
 \mathcal T_K|_{\mathbb P_L}=\id\right\}.
\end{equation}
Here bandwidth at most \(N\) means that \(K\) has the finite expansion
\eqref{eq:zonal-kernel}.
The following proposition reduces the generalized-projection problem
\eqref{eq:intro-lambda} to the minimization of positivity loss over
zonal operators.
\begin{proposition}[Symmetry reduction]\label{prop:rotation}
For \(0\le L\le N\),
\[
 \lambda_d(L,N)
 =\inf_{\mathcal T_K\in\mathfrak T_d(L,N)}
   \norm{\mathcal T_K}_{C(\Sd)\to C(\Sd)}
 =1+2\inf_{\mathcal T_K\in\mathfrak T_d(L,N)}
       \mathfrak p(\mathcal T_K).
\]
\end{proposition}

\begin{proof}
Since every operator in \(\mathfrak T_d(L,N)\) is admissible in
\eqref{eq:intro-lambda}, we have
\[
 \lambda_d(L,N)
 \le\inf_{\mathcal T_K\in\mathfrak T_d(L,N)}\norm{\mathcal T_K}.
\]
For the reverse inequality, we fix an arbitrary operator \(\mathcal A\)
admissible in \eqref{eq:intro-lambda}.  Let \(G:=O(d+1)\). Each
\(g\in G\) acts isometrically on \(C(\Sd)\) by
\((\mathcal R_gf)(x):=f(g^{-1}x)\),
and we average the conjugates of \(\mathcal A\) by setting
\[
 \overline{\mathcal A}f
 :=\int_G\mathcal R_g^{-1}\mathcal A\mathcal R_gf\,\dd g
\]
for $f\in C(\Sd)$, where \(\dd g\) is the normalized Haar probability measure on \(G\).  For
fixed \(f\), the integrand
is a continuous \(\mathbb P_N\)-valued function of \(g\), so this
finite-dimensional Bochner integral is well defined.
Since \(\mathcal R_g\) is an isometry on \(C(\Sd)\), and hence
\(\|\mathcal R_g^{-1}\mathcal A\mathcal R_g\|=\|\mathcal A\|\),
we have
\begin{equation*}
 \|\overline{\mathcal A}f\|_\infty\le
 \int_G
 \|\mathcal R_g^{-1}\mathcal A\mathcal R_gf\|_\infty
 \,\dd g =
 \int_G
 \|\mathcal A\mathcal R_gf\|_\infty
 \,\dd g \le
 \|\mathcal A\|
 \int_G\|\mathcal R_gf\|_\infty\,\dd g
 =
 \|\mathcal A\|\,\|f\|_\infty.
\end{equation*}
Here the last equality uses \(\int_G1\,\dd g=1\).  Thus
\(\norm{\overline{\mathcal A}}\le\norm{\mathcal A}\).  Since
\(\mathcal R_g^{-1}\mathcal A\mathcal R_gf\in\mathbb P_N\), after averaging
we still obtain \(\ran\overline{\mathcal A}\subseteq\mathbb P_N\).
Since \(\mathcal R_g\) preserves \(\mathbb P_L\) and \(\mathcal A\) reproduces
\(\mathbb P_L\), for every \(p\in\mathbb P_L\) and \(g\in G\),
\[\mathcal R_g^{-1}\mathcal A\mathcal R_gp
 =\mathcal R_g^{-1}\mathcal R_gp=p.\]
After averaging, we have \(\overline{\mathcal A}p=p\), implying that $\overline{\mathcal A}$ also reproduces
\(\mathbb P_L\).
For $h\in G$, since
\[\mathcal{R}_h^{-1} \overline{\mathcal{A}} \mathcal{R}_h
 =\int_G \mathcal{R}_h^{-1} \mathcal{R}_g^{-1} \mathcal{A}
 \mathcal{R}_g \mathcal{R}_h\,\dd g
 =\int_G \mathcal{R}_{gh}^{-1} \mathcal{A} \mathcal{R}_{gh}\,\dd g,\]
the right invariance of Haar measure under the change of variables \(u=gh\)
yields \(\mathcal R_h^{-1}\overline{\mathcal A}\mathcal R_h
=\overline{\mathcal A}\), that is,
\(\overline{\mathcal A}\mathcal R_h
 =\mathcal R_h\overline{\mathcal A}\). Thus \(\overline{\mathcal A}\) is \(O(d+1)\)-equivariant.
Since each \(\Proj_\ell\) is also \(O(d+1)\)-equivariant,
for \(j\ge0\) and \(0\le\ell\le N\), the map
\[
 \Proj_\ell\overline{\mathcal A}|_{\mathbb H_j}:
 \mathbb H_j\longrightarrow\mathbb H_\ell
\]
intertwines the orthogonal action. Indeed, for \(d\ge2\), the complexified harmonic
spaces \(\mathbb H_\ell^{\mathbb C}\) are pairwise inequivalent irreducible
\(O(d+1)\)-representations.  After complexification, Schur's lemma therefore
implies that the above map vanishes for \(j\ne\ell\) and is a scalar
multiple of the identity for \(j=\ell\).  Since the original operator is
real, the scalar is real.  For \(d=1\), the complexified spaces are likewise
pairwise inequivalent irreducible \(O(2)\)-representations: a reflection
interchanges the two rotation weights.  The same application of Schur's
lemma therefore gives the conclusion.
Thus \(\overline{\mathcal A}\) does not mix harmonic degrees.  Since
\(\ran\overline{\mathcal A}\subseteq\mathbb P_N\), there are real numbers
\(a_0,\ldots,a_N\) such that, for every spherical polynomial \(f\),
\begin{equation}\label{eq:bothsides}
 \overline{\mathcal A}f
 =\sum_{\ell=0}^N a_\ell\Proj_\ell f.
\end{equation}
Since spherical polynomials are dense in \(C(\Sd)\) and both sides of
\eqref{eq:bothsides} define bounded operators on \(C(\Sd)\), the expression
\eqref{eq:bothsides} extends to every \(f\in C(\Sd)\).  By
\eqref{eq:harmonic-projection},
\(
 \overline{\mathcal A}f(x)
 =\int_{\Sd}K_{\overline{\mathcal A}}(x,y)f(y)\,\dd\sigma(y),
\)
where
\[ K_{\overline{\mathcal A}}(x,y)
 :=\sum_{\ell=0}^N a_\ell\mathcal Z_\ell(x,y).\]
The addition formula \eqref{eq:addition-formula} shows that
\(K_{\overline{\mathcal A}}(x,y)\) depends only on \(x\cdot y\) and is
therefore zonal.  Together with the range and reproduction properties established
above, this gives
$\overline{\mathcal A}\in\mathfrak T_d(L,N)$.
Therefore
\[
 \inf_{\mathcal T_K\in\mathfrak T_d(L,N)}\norm{\mathcal T_K}
 \le\norm{\overline{\mathcal A}}
 \le\norm{\mathcal A}.
\]
Since \(\mathcal A\) is arbitrary, taking the infimum over all admissible
operators proves the reverse inequality.  The final equality follows from
\eqref{eq:norm-negative-mass}.
\end{proof}

By Proposition~\ref{prop:rotation}, the lower bound for
\(\lambda_d(L,N)-1\) now reduces to a uniform lower bound for the positivity
loss over \(\mathfrak T_d(L,N)\).  The argument uses a low-frequency peak
polynomial and the reproducing kernel of \(\mathbb P_N\).

\begin{lemma}[Fej\'er peak]\label{lem:fejer-peak}
For \(L\ge1\), let
\begin{equation}\label{eq:R-L}
 R_L(\cos\theta)
 :=\frac{1}{(L+1)^2}
 \left(
 \frac{\sin((L+1)\theta/2)}{\sin(\theta/2)}
 \right)^2,
 \qquad 0\le\theta\le\pi.
\end{equation}
For \(e\in\Sd\), let 
\[R_{L,e}(y):=R_L(e\cdot y)\] 
and 
\[q_{L,e}(y):=1-R_{L,e}(y).\]
Then \(R_{L,e}\in\mathbb P_L\), $0\le R_{L,e}\le1$, and $R_{L,e}(e)=1$. Moreover,
 \(q_{L,e}\in\mathbb P_L\), $0\le q_{L,e}\le1$, $q_{L,e}(e)=0$, and there is an absolute constant \(c_0>0\) such that
\begin{equation}\label{eq:q-peak}
 q_{L,e}(y)
 \ge c_0\min\{1,L^2\rho(e,y)^2\},
 \qquad y\in\Sd.
\end{equation}
\end{lemma}
\begin{proof}
The finite geometric sum and the Fej\'er identity give
\[
 \left(
  \frac{\sin((L+1)\theta/2)}{\sin(\theta/2)}
 \right)^2
 =
 \left|\sum_{j=0}^L e^{ij\theta}\right|^2
 =
 (L+1)+2\sum_{k=1}^L(L+1-k)\cos(k\theta),
\]
where the value at \(\theta=0\) is understood by continuity.  Since
\(\cos(k\theta)=T_k(\cos\theta)\), where \(T_k\) is the Chebyshev polynomial of
degree \(k\),  it follows that \(R_L\) is an algebraic polynomial of degree at most \(L\).
Hence \(R_{L,e}\in\mathbb P_L\) and \(q_{L,e}\in\mathbb P_L\).  Since
\[
 0\le R_L(\cos\theta)
 =\frac{1}{(L+1)^2}
  \left|\sum_{j=0}^L e^{ij\theta}\right|^2
 \le1,
\]
and \(R_L(1)=1\), it immediately follows that
\(0\le R_{L,e}\le1\), \(R_{L,e}(e)=1\), \(0\le q_{L,e}\le1\), and
\(q_{L,e}(e)=0\).
It remains to prove \eqref{eq:q-peak}.  Let
$m:=L+1$, $x:={\theta}/{2}$, and $A_m(x):={\sin(mx)}/{(m\sin x)}$.
Then
\[
 R_L(\cos\theta)=A_m(x)^2.
\]
Let \(U\) be uniformly distributed on
$\{m-1,m-3,\ldots,-m+1\}$.
The finite geometric sum gives
\[
 A_m(x)
 =\frac1m\sum_{j=0}^{m-1}e^{i(m-1-2j)x}
 =\E e^{iUx}.
\]
If \(V\) is an independent copy of \(U\), then
$|A_m(x)|^2
 =\E e^{i(U-V)x}$.
Since \(U-V\) has a symmetric distribution, we have
\begin{equation}\label{eq:variance-fejer}
 1-|A_m(x)|^2
 =\E\bigl[1-\cos((U-V)x)\bigr].
\end{equation}
Suppose first that \(mx\le2\).  Since
\(|U-V|\le2(m-1)\), we have \(|(U-V)x|\le4\).  Hence
$1-\cos u\ge c u^2$ for $|u|\le4$ and an absolute constant \(c>0\). Together with $\E U=0$ and
\[
 \operatorname{Var}(U)
 =\frac1m\sum_{j=0}^{m-1}(m-1-2j)^2
 =\frac{m^2-1}{3},
\]
 \eqref{eq:variance-fejer} yields
\[
 1-|A_m(x)|^2
 \ge c x^2\E(U-V)^2
 =\frac{2c(m^2-1)}{3}x^2
 \ge c' m^2x^2,
\]
since \(m\ge2\).
If \(mx\ge2\), then \(0\le x\le\pi/2\) and
\(\sin x\ge2x/\pi\), implying
\[
 |A_m(x)|
 \le\frac{1}{m\sin x}
 \le\frac{\pi}{2mx}
 \le\frac{\pi}{4}.
\]
Thus
\[
 1-|A_m(x)|^2\ge1-\frac{\pi^2}{16}.
\]
Combining the two cases,
\[
 1-|A_m(x)|^2
 \ge c\min\{1,m^2x^2\}.
\]
Recalling \(m=L+1\) and \(x=\theta/2\), and letting
\(\theta=\rho(e,y)\), we have \eqref{eq:q-peak}.
\end{proof}

We then show that a polynomial kernel of bandwidth \(N\) cannot concentrate
all of its positive mass below the spectral scale \(N^{-1}\). To prove this, we use the following standard \(L^1\)-Nikolskii estimate in a form with an
explicit constant.

\begin{lemma}[\(L^1\)-Nikolskii estimate]
\label{lem:L1-Nikolskii}
For every \(P\in\mathbb P_N\),
\[\norm{P}_\infty\le D_{d,N}\norm{P}_{L^1(\Sd)},\]
where \(D_{d,N}=\dim\mathbb P_N\).
\end{lemma}

\begin{proof}
Let $\mathcal D_N(x,y)
 :=\sum_{\ell=0}^{N}\mathcal Z_\ell(x,y)$,
the reproducing kernel of \(\mathbb P_N\). Thus
\begin{equation}\label{eq:PN-reproduction}
 P(x)=\int_{\Sd}\mathcal D_N(x,y)P(y)\,\dd\sigma(y)
 \qquad \forall P\in\mathbb P_N.
\end{equation}
By the addition formula \eqref{eq:addition-formula},
\[\mathcal D_N(x,x)
 =\sum_{\ell=0}^N Z(d,\ell)
 =D_{d,N}.\]
Using the orthonormal-basis expansion of \(\mathcal D_N\) and
Cauchy--Schwarz,
\[
 \abs{\mathcal D_N(x,y)}
 \le
 \sqrt{\mathcal D_N(x,x)\mathcal D_N(y,y)}
 =D_{d,N}.
\]
Hence \eqref{eq:PN-reproduction} gives
\[
 \abs{P(x)}
 \le
 D_{d,N}\int_{\Sd}\abs{P(y)}\,\dd\sigma(y)
 =
 D_{d,N}\norm{P}_{L^1(\Sd)}.
\]
Taking the supremum over \(x\in\Sd\) proves
the lemma.
\end{proof}

Let $B(e,r):=\{y\in\Sd:\rho(e,y)\le r\}$ be the spherical cap with radius $r$ centered at $e$.
There are constants \(c_d,C_d>0\) such that
\begin{equation}\label{eq:cap-volume}
 c_dr^d\le\sigma(B(e,r))\le C_dr^d,
 \qquad 0<r\le1.
\end{equation}
Since \(D_{d,N}\asymp_d(N+1)^d\), a cap of measure
\(D_{d,N}^{-1}\) has radius comparable to \((N+1)^{-1}\), the natural
spatial scale for bandwidth \(N\).  Combined with
Lemma~\ref{lem:L1-Nikolskii}, this prevents a normalized bandlimited polynomial
from concentrating all of its positive mass inside such a cap.
\begin{lemma}[Positive mass outside the spectral scale]
\label{lem:mass-outside}
Fix \(e\in\Sd\), and let \(P\in\mathbb P_N\) be real-valued with
$\int_{\Sd}P\,\dd\sigma=1$.
Set $\Lambda:=\norm{P}_{L^1(\Sd)}$,
and let \(r_N\in(0,\pi)\) be determined by
\[
 \sigma(B(e,r_N))=\frac{1}{4D_{d,N}}.
\]
Then $r_N\ge c_d(N+1)^{-1}$,
and
\begin{equation}\label{eq:positive-mass-outside}
 \int_{\Sd\setminus B(e,r_N)}P_+(y)\,\dd\sigma(y)
 \ge\frac{\Lambda+2}{4}
 \ge\frac34.
\end{equation}
\end{lemma}

\begin{proof}
If \(r_N\le1\), then \eqref{eq:cap-volume} and
\(D_{d,N}\asymp_d(N+1)^d\) give
\[
 \frac{1}{4D_{d,N}}
 =\sigma(B(e,r_N))
 \le C_dr_N^d,
\]
hence \(r_N\ge c_d(N+1)^{-1}\).  If \(r_N>1\), the same estimate follows
after decreasing \(c_d\).
By Lemma~\ref{lem:L1-Nikolskii},
$\norm{P}_\infty\le D_{d,N}\Lambda$.
Therefore
\[
 \int_{B(e,r_N)}P_+\,\dd\sigma
 \le \norm{P}_\infty\sigma(B(e,r_N))
 \le\frac{\Lambda}{4}.
\]
On the other hand,
since $\int_{\Sd}P_+\,\dd\sigma-\int_{\Sd}P_-\,\dd\sigma=1$
and $\int_{\Sd}P_+\,\dd\sigma+\int_{\Sd}P_-\,\dd\sigma=\Lambda$,
we have
\[
 \int_{\Sd}P_+\,\dd\sigma=\frac{\Lambda+1}{2}.
\]
It follows that
\[
 \int_{\Sd\setminus B(e,r_N)}P_+\,\dd\sigma
 \ge
 \frac{\Lambda+1}{2}-\frac{\Lambda}{4}
 =
 \frac{\Lambda+2}{4}.
\]
Finally, since $\Lambda=\norm{P}_{L^1(\Sd)}
 \ge\left|\int_{\Sd}P\,\dd\sigma\right|=1$, we have the last inequality in
\eqref{eq:positive-mass-outside}.
\end{proof}

We now combine the peak and concentration estimates to establish the lower bound.

\begin{theorem}[Lower bound]\label{thm:lower}
There exists \(c_d>0\), depending only on \(d\), such that for every
\(1\le L\le N\) and every operator \(\mathcal A\) admissible in
\eqref{eq:intro-lambda}:
\begin{equation}\label{eq:individual-operator-lower}
 \mathfrak p(\mathcal A)
 \ge \frac{c_d}{2}\left(\frac{L}{N+1}\right)^2,
\end{equation}
and
\begin{equation}\label{eq:lower-lambda}
 \lambda_d(L,N)-1
 \ge c_d\left(\frac{L}{N+1}\right)^2.
\end{equation}
\end{theorem}
\begin{proof}
Fix \(\mathcal T_K\in\mathfrak T_d(L,N)\) and \(e\in\Sd\).  Let
\[
 K_+(y):=\max\{K(e,y),0\}\]
 and 
 \[
 K_-(y):=\max\{-K(e,y),0\}
\]
for the positive and negative parts of \(K(e,\cdot)\), respectively, and let
\[
 m_-:=\int_{\Sd}K_-(y)\,\dd\sigma(y)
 =\mathfrak p(\mathcal T_K).
\]
Let \(q_{L,e}\) be given by Lemma~\ref{lem:fejer-peak}.  Since
\(q_{L,e}\in\mathbb P_L\), \(q_{L,e}(e)=0\), and
\(\mathcal T_K\) reproduces \(\mathbb P_L\), we have
\[
 0
 =\mathcal T_Kq_{L,e}(e)
 =\int_{\Sd}q_{L,e}(y)
   \bigl(K_+(y)-K_-(y)\bigr)\,\dd\sigma(y).
\]
Hence
\begin{equation}\label{eq:balance}
 \int_{\Sd}q_{L,e}K_+\,\dd\sigma
 =
 \int_{\Sd}q_{L,e}K_-\,\dd\sigma
 \le m_-,
\end{equation}
since \(0\le q_{L,e}\le1\). Now since \(K(e,\cdot)\in\mathbb P_N\) and
$\int_{\Sd}K(e,y)\,\dd\sigma(y)
 =\mathcal T_K1(e)
 =1$,
Lemma~\ref{lem:mass-outside}, applied to \(P=K(e,\cdot)\), gives a radius
\(r_N\ge c_d(N+1)^{-1}\) such that
\[
 \int_{\Sd\setminus B(e,r_N)}K_+\,\dd\sigma\ge\frac34.
\]
On \(\Sd\setminus B(e,r_N)\), the Fej\'er peak estimate
\eqref{eq:q-peak} gives
$q_{L,e}(y)
 \ge c_0\min\{1,L^2r_N^2\}$.
Therefore
\[
 \int_{\Sd}q_{L,e}K_+\,\dd\sigma
 \ge
 \frac{3c_0}{4}\min\{1,L^2r_N^2\}.
\]
Since \(r_N\ge c_d(N+1)^{-1}\) and \(L\le N\), after decreasing \(c_d\)
if necessary,
\begin{equation}\label{eq:weighted-positive-lower}
 \int_{\Sd}q_{L,e}K_+\,\dd\sigma
 \ge
 \frac{c_d}{2}\left(\frac{L}{N+1}\right)^2.
\end{equation}
Combining \eqref{eq:balance} and
\eqref{eq:weighted-positive-lower} proves that 
for every
\(1\le L\le N\) and every
\(\mathcal T_K\in\mathfrak T_d(L,N)\),
\begin{equation}\label{eq:individual-kernel-lower}
 \mathfrak p(\mathcal T_K)
 \ge \frac{c_d}{2}\left(\frac{L}{N+1}\right)^2,
 \qquad
 \norm{\mathcal T_K}_{C(\Sd)\to C(\Sd)}-1
 =2\mathfrak p(\mathcal T_K)
 \ge c_d\left(\frac{L}{N+1}\right)^2.
\end{equation}
If \(\mathcal A\) is admissible in \eqref{eq:intro-lambda}, let
\(\overline{\mathcal A}\in\mathfrak T_d(L,N)\) be its rotational average
from Proposition~\ref{prop:rotation}.  Since
\(\norm{\overline{\mathcal A}}\le\norm{\mathcal A}\),
\eqref{eq:general-norm-defect} and
\eqref{eq:individual-kernel-lower} give
\[
 2\mathfrak p(\mathcal A)
 =\norm{\mathcal A}-1
 \ge\norm{\overline{\mathcal A}}-1
 \ge c_d\left(\frac{L}{N+1}\right)^2.
\]
This proves \eqref{eq:individual-operator-lower}.  Finally, taking the
infimum over all operators admissible in \eqref{eq:intro-lambda} gives
\eqref{eq:lower-lambda}.
\end{proof}

\section{The matching upper bound}
\label{sec:upper}

We prove the matching upper bound by constructing, for
\(1\le L\le\kappa N\) with \(0<\kappa<1\), a bandwidth-\(N\) operator
that reproduces \(\mathbb P_L\) exactly and satisfies the desired upper bound.
The construction starts from a positive operator \(\mathcal J_N\) with
range in \(\mathbb P_N\).  Although \(\mathcal J_N\) need not reproduce
nonconstant polynomials exactly, its error on \(\mathbb P_m\) is of order
\((m/(N+1))^2\).  A flat-top correction then restores exact reproduction
without changing this order.  We use a classical spherical Jackson operator
for \(\mathcal J_N\); see, e.g.,
\cite{DaiXu2013,LizorkinNikolskii1983}.

\begin{lemma}[Positive Jackson approximation]
\label{lem:positive-Jackson}
For every \(d\ge1\), there exist positive zonal operators
$\mathcal J_N:C(\Sd)\to\mathbb P_N$ for $N\ge1$
and a constant \(\gamma_d^{\mathrm J}>0\) such that
\begin{enumerate}
\item[(i)] \(\mathcal J_N1=1\), \(\mathcal J_Nf\ge0\) whenever \(f\ge0\), and
      \(\norm{\mathcal J_N}_{C(\Sd)\to C(\Sd)}=1\);
\item[(ii)] for every \(p\in\mathbb P_m\), \(0\le m\le N\),
\begin{equation}\label{eq:Jackson-poly-error}
 \norm{(\id-\mathcal J_N)p}_\infty
 \le \gamma_d^{\mathrm J}\left(\frac{m}{N+1}\right)^2\norm{p}_\infty.
\end{equation}
\end{enumerate}
\end{lemma}

\begin{proof}
For $s:=\left\lfloor{(d+2)}/{2}\right\rfloor+1$ and \(k\ge2\), we define the spherical Jackson kernel
\[
 J_{k,s}(\cos\theta)
 :=c_{k,s,d}
 \left(\frac{\sin(k\theta/2)}{\sin(\theta/2)}\right)^{2s},
 \qquad 0\le\theta\le\pi,
\]
where \(c_{k,s,d}>0\) is chosen such that the normalization $\int_{\Sd}J_{k,s}(x\cdot y)\,\dd\sigma(y)=1$ holds.
By rotational invariance, the normalizing constant is independent of
\(x\).  The Fej\'er identity
\[
 \left(\frac{\sin(k\theta/2)}{\sin(\theta/2)}\right)^2
 =k+2\sum_{j=1}^{k-1}(k-j)\cos(j\theta)
\]
shows that \(J_{k,s}(\cos\theta)\) is a nonnegative algebraic polynomial
in \(\cos\theta\) of degree at most \(s(k-1)\).  Hence the corresponding Jackson operator, a zonal operator,
\[
 \mathcal J_{k,s}f(x)
 :=\int_{\Sd}J_{k,s}(x\cdot y)f(y)\,\dd\sigma(y)
\]
has range in \(\mathbb P_{s(k-1)}\).  Moreover, $\mathcal J_{k,s}1=1$ and
$\mathcal J_{k,s}f\ge0$ whenever $f\ge0$.
Since
\[
 \abs{\mathcal J_{k,s}f(x)}
 \le \norm f_\infty
     \int_{\Sd}J_{k,s}(x\cdot y)\,\dd\sigma(y)
 =\norm f_\infty,
\]
we have \(\norm{\mathcal J_{k,s}}\le1\), while
\(\mathcal J_{k,s}1=1\) gives the reverse inequality.  Therefore $\norm{\mathcal J_{k,s}}_{C(\Sd)\to C(\Sd)}=1$.

We next establish the second-moment estimate
\begin{equation}\label{eq:Jackson-moment}
 \sup_{x\in\Sd}
 \int_{\Sd}\rho(x,y)^2J_{k,s}(x\cdot y)\,\dd\sigma(y)
 \lesssim_d k^{-2}.
\end{equation}
Let
\[ \Phi_{k,s}(\theta)
 := \left(\frac{\sin(k\theta/2)}{\sin(\theta/2)}\right)^{2s}.\]
Spherical polar coordinates give
\[
 \int_{\Sd}\rho(x,y)^2J_{k,s}(x\cdot y)\,\dd\sigma(y)
 =
 \frac{
 \displaystyle
 \int_0^\pi
 \theta^2\Phi_{k,s}(\theta)(\sin\theta)^{d-1}\,\dd\theta}
 {\displaystyle
 \int_0^\pi
 \Phi_{k,s}(\theta)(\sin\theta)^{d-1}\,\dd\theta}.
\]
If \(0\le\theta\le k^{-1}\), the elementary inequalities $\sin(k\theta/2)\ge{k\theta}/{\pi}$
and $\sin(\theta/2)\le{\theta}/{2}$ lead to $\Phi_{k,s}(\theta)\gtrsim_s k^{2s}$.
Since \(\sin\theta\asymp\theta\) on \([0,1]\),
\[
 \int_0^\pi
 \Phi_{k,s}(\theta)(\sin\theta)^{d-1}\,\dd\theta
 \gtrsim_{d,s}
 k^{2s}\int_0^{k^{-1}}\theta^{d-1}\,\dd\theta
 \asymp_{d,s}k^{2s-d}.
\]
If \(0<\theta\le\pi\), the elementary inequalities
$|\sin(k\theta/2)|\le k|\sin(\theta/2)|$,
$\sin(\theta/2)\ge{\theta}/{\pi}$,
and $|\sin(k\theta/2)|\le1$
give $\Phi_{k,s}(\theta)
 \lesssim_s\min\{k^{2s},\theta^{-2s}\}$.
Splitting the integral over
\([0,k^{-1}]\), \([k^{-1},1]\), and \([1,\pi]\), we obtain
\[
 \int_0^\pi
 \theta^2\Phi_{k,s}(\theta)(\sin\theta)^{d-1}\,\dd\theta
 \lesssim_{d,s}
 k^{2s}\int_0^{k^{-1}}\theta^{d+1}\,\dd\theta +\int_{k^{-1}}^1\theta^{d+1-2s}\,\dd\theta+1\lesssim_{d,s}k^{2s-d-2},
\]
where \(2s>d+2\) is used in the last step.  Taking the ratio proves
\eqref{eq:Jackson-moment}.

We now apply \eqref{eq:Jackson-moment} to \(p\in\mathbb P_m\).
Fix \(x\in\Sd\), and let \(\mathbb S_x^{d-1}\) be the unit sphere in
the tangent space \(T_x\Sd\), equipped with normalized surface measure
\(\sigma_{d-1,x}\).  For \(\xi\in\mathbb S_x^{d-1}\), let $g_\xi(t):=p(\cos t\,x+\sin t\,\xi)$,
 a trigonometric polynomial of degree at most
\(m\).  The Bernstein inequality therefore gives
\[
 \norm{g_\xi''}_\infty
 \le m^2\norm{g_\xi}_\infty
 \le m^2\norm p_\infty.
\]
Taylor's formula gives 
\[g_\xi(\theta)-g_\xi(0)
 =\theta g_\xi'(0)+r_\xi(\theta),\] where
 $\abs{r_\xi(\theta)}
 \le\frac12m^2\theta^2\norm p_\infty$.
By the chain rule, $g_\xi'(0)=\nabla_{\Sd}p(x)\cdot\xi$, which implies
\(g_{-\xi}'(0)=-g_\xi'(0)\).  By the symmetry of
\(\sigma_{d-1,x}\),
\[
 \int_{\mathbb S_x^{d-1}}g_\xi'(0)\,
 \dd\sigma_{d-1,x}(\xi)=0.
\]
It follows that
\begin{equation}\label{eq:Jackson-direction-average}
 \left|
 \int_{\mathbb S_x^{d-1}}
 \bigl(g_\xi(\theta)-g_\xi(0)\bigr)
 \,\dd\sigma_{d-1,x}(\xi)
 \right|
 \le\frac12m^2\theta^2\norm p_\infty.
\end{equation}
Since \(J_{k,s}\) is normalized,
\[
 \mathcal J_{k,s}p(x)-p(x)
 =
 \int_{\Sd}J_{k,s}(x\cdot y)
 \bigl(p(y)-p(x)\bigr)\,\dd\sigma(y).
\]
For $y=\cos\theta\,x+\sin\theta\,\xi$, where  $\xi\in\mathbb S_x^{d-1}$,
we have $p(y)-p(x)=g_\xi(\theta)-g_\xi(0)$ and 
$\rho(x,y)=\theta$.
Since \(J_{k,s}(x\cdot y)=J_{k,s}(\cos\theta)\) is independent of
\(\xi\), spherical polar coordinates give
\[
 \mathcal J_{k,s}p(x)-p(x)
 =
 \alpha_d\int_0^\pi J_{k,s}(\cos\theta)
 \left[
 \int_{\mathbb S_x^{d-1}}
 \bigl(g_\xi(\theta)-g_\xi(0)\bigr)
 \,\dd\sigma_{d-1,x}(\xi)
 \right] 
 (\sin\theta)^{d-1}\,\dd\theta,
\]
where
\(
 \alpha_d:=\Gamma((d+1)/2)/(\sqrt{\pi}\,\Gamma(d/2))
\)
is the normalization constant in the spherical polar-coordinate formula.
Using \eqref{eq:Jackson-direction-average} and the nonnegativity of
\(J_{k,s}\), we obtain
\[
 \abs{\mathcal J_{k,s}p(x)-p(x)}
 \le
 \frac{m^2\norm p_\infty}{2}
 \int_{\Sd}
 \rho(x,y)^2J_{k,s}(x\cdot y)\,\dd\sigma(y).
\]
By the second-moment estimate \eqref{eq:Jackson-moment},
\[
 \abs{\mathcal J_{k,s}p(x)-p(x)}
 \lesssim_d
 \left(\frac{m}{k}\right)^2\norm p_\infty.
\]
Taking the supremum over \(x\in\Sd\) yields
\begin{equation}\label{eq:keyestimate}
 \norm{(\id-\mathcal J_{k,s})p}_\infty
 \le C_d\left(\frac{m}{k}\right)^2\norm p_\infty.
\end{equation}
To show \eqref{eq:Jackson-poly-error}, we first consider \(N\ge2s\), and choose
 $k:=\left\lfloor N/s\right\rfloor+1$.
Then $s(k-1)\le N$ and $k\asymp_d N+1$.
Hence
\(\mathcal J_N:=\mathcal J_{k,s}\) has range in \(\mathbb P_N\), and
\eqref{eq:keyestimate} gives \eqref{eq:Jackson-poly-error}. We then consider the remaining case
of  \(1\le N<2s\).  Let
\(\mathcal J_N:=\Proj_0\), the positive projection onto the constants.
If \(m=0\), then
\((\id-\Proj_0)p=0\).  If \(1\le m\le N<2s\), then
\[
 \left(\frac{m}{N+1}\right)^2
 \ge \left(\frac{1}{N+1}\right)^2
 \ge\frac{1}{(2s)^2},
\]
while $\norm{(\id-\Proj_0)p}_\infty
 \le2\norm p_\infty$.
Increasing \(\gamma_d^{\mathrm J}\), if necessary, covers these finitely
many values of \(N\) and completes the proof.
\end{proof}

We also use a smooth flat-top filtered approximation operator.
\begin{lemma}[Filtered approximation]
\label{lem:flat-top}
Fix \(\beta>1\). There exists a real \(C^\infty\) filter
\(h_\beta:[0,\infty)\to[0,1]\) satisfying
\begin{equation*}
h_\beta(t) = \begin{cases}
1, & 0\le t\le1,\\
0, & t\ge \beta.
\end{cases}
\end{equation*}
The zonal filtered approximation operators
\begin{equation}\label{eq:V-L}
 \mathcal V_Lf
 :=\sum_{\ell\ge0}h_\beta\left(\frac{\ell}{L}\right)\Proj_\ell f
\end{equation}
reproduce \(\mathbb P_L\), have range in
\(\mathbb P_{\lceil\beta L\rceil-1}\), and satisfy
\begin{equation}\label{eq:flat-top-properties}
 \sup_{L\ge1}
 \norm{\mathcal V_L}_{C(\Sd)\to C(\Sd)}
 \le B_{d,\beta}.
\end{equation}
\end{lemma}
\begin{proof}
The reproduction and range properties follow directly from the support
conditions on \(h_\beta\). 
The uniform bound \eqref{eq:flat-top-properties}
is the standard boundedness theorem for smooth filtered approximation
operators; see \cite{SloanWomersley2012,WangSloan2017}.
\end{proof}

We now combine the positive Jackson operator $\mathcal{J}_N$ with the filtered approximation operator $\mathcal{V}_L$
to obtain an exactly reproducing operator with the matching excess norm.

\begin{theorem}[Upper bound]
\label{thm:upper}
Fix \(0<\kappa<1\).  There exists \(C_{d,\kappa}>0\) such that, for every
\(1\le L\le\kappa N\), the zonal operator
\begin{equation}\label{eq:T-construction}
 \mathcal T_{L,N}
 :=\mathcal J_N+(\id-\mathcal J_N)\mathcal V_L
\end{equation}
belongs to \(\mathfrak T_d(L,N)\) and satisfies
\begin{equation}\label{eq:individual-kernel-upper}
 \mathfrak p(\mathcal T_{L,N})
 \le \frac{C_{d,\kappa}}{2}\left(\frac{L}{N+1}\right)^2.
\end{equation}
Consequently,
\begin{equation}\label{eq:upper-lambda}
 \lambda_d(L,N)-1
 \le C_{d,\kappa}\left(\frac{L}{N+1}\right)^2.
\end{equation}
\end{theorem}

\begin{proof}
Choose \(1<\beta<\kappa^{-1}\), set
\(m_L:=\lceil\beta L\rceil-1\), and take \(\mathcal V_L\) and
\(\mathcal J_N\) from Lemmas~\ref{lem:flat-top} and~
\ref{lem:positive-Jackson}, respectively.  Since \(L\le\kappa N\),
\[
 m_L\le\beta L\le\beta\kappa N<N,
 \qquad
 \ran\mathcal V_L\subseteq\mathbb P_{m_L}\subseteq\mathbb P_N.
\]
The range, zonality, and reproduction properties in these two lemmas,
together with \eqref{eq:T-construction}, show that
\(\mathcal T_{L,N}\in\mathfrak T_d(L,N)\).

For \(f\in C(\Sd)\), the polynomial \(\mathcal V_Lf\) belongs to
\(\mathbb P_{m_L}\).  Applying \eqref{eq:Jackson-poly-error} to this
polynomial and then using \eqref{eq:flat-top-properties}, we obtain
\[
 \norm{(\id-\mathcal J_N)\mathcal V_Lf}_\infty
 \le \gamma_d^{\mathrm J}\left(\frac{m_L}{N+1}\right)^2
       \norm{\mathcal V_Lf}_\infty
 \le C_{d,\kappa}\left(\frac{L}{N+1}\right)^2\norm f_\infty.
\]
Since \(\norm{\mathcal J_N}=1\), it follows from
\eqref{eq:T-construction} and \eqref{eq:norm-negative-mass} that
\[
 2\mathfrak p(\mathcal T_{L,N})
 =\norm{\mathcal T_{L,N}}-1
 \le\norm{(\id-\mathcal J_N)\mathcal V_L}
 \le C_{d,\kappa}\left(\frac{L}{N+1}\right)^2.
\]
This proves \eqref{eq:individual-kernel-upper}.  Finally,
\(\lambda_d(L,N)\le\norm{\mathcal T_{L,N}}\), and hence
\eqref{eq:upper-lambda} follows.
\end{proof}

Combining Theorems~\ref{thm:lower} and~\ref{thm:upper}
proves Theorem~\ref{thm:intro}.

\section{Applications and discussion}
\label{sec:consequences}

Theorem~\ref{thm:intro} gives a quantitative form of the
impossibility triangle formed by positivity, exact low-frequency reproduction,
and finite spectral bandwidth.  Any two of these properties are compatible,
but all three cannot hold simultaneously.  This trade-off already appears
in classical Fourier summability. For example, Fej\'er means preserve positivity and
finite bandwidth at the expense of exact reproduction, while delayed
de la Vall\'ee--Poussin means restore exact low-frequency reproduction
through a transition band and are no longer positive in general.  The
operators \(\mathcal J_N\) and \(\mathcal V_L\) used above play analogous
roles on the sphere.  We now discuss three consequences of the quadratic
obstruction, beginning with its relation to classical Fourier summability
on the circle.

\subsection{Generalized projections and de la Vall\'ee--Poussin means}

Let \(\mathbb T:=\mathbb R/(2\pi\mathbb Z)\), and identify \(C(\mathbb T)\)
with the space of real-valued continuous \(2\pi\)-periodic functions. For
\(n\ge0\), let
\[
 \mathbb P_n(\mathbb T)
 :=
 \left\{
 a_0+\sum_{k=1}^{n}
 \bigl(a_k\cos(k\theta)+b_k\sin(k\theta)\bigr):
 a_k,b_k\in\mathbb R
 \right\}
\]
be the space of real trigonometric polynomials of degree at most \(n\).
Under the parametrization
\(\theta\mapsto(\cos\theta,\sin\theta)\),
the space \(\mathbb P_n(\mathbb T)\) is naturally identified with
\(\mathbb P_n(\mathbb S^1)\). Theorem~\ref{thm:intro} therefore gives
the following two-scale estimate.

\begin{corollary}\label{cor:circle}
For all integers \(s\ge3\) and \(L\ge1\),
\begin{equation}\label{eq:circle}
 \lambda_1(L,sL-1)-1
 \asymp s^{-2},
\end{equation}
uniformly in \(s\) and \(L\).
\end{corollary}

\begin{proof}
Let \(N=sL-1\).  Since \(s\ge3\) and \(L\ge1\), we have
\[
 L\le\frac12(sL-1)=\frac N2.
\]
Together with ${L}/{(N+1)}=1/s$, the result follows from Theorem~\ref{thm:intro} with
\(d=1\) and \(\kappa=1/2\).
\end{proof}

On the circle, the impossibility triangle is already visible in classical Fourier
summability. Fej\'er means are positive bandlimited contractions, but they
reproduce only the constants exactly. Delayed de la Vall\'ee--Poussin means
use a transition band to recover exact reproduction of a prescribed
low-frequency space, at the cost of positivity in general; see, e.g.,
\cite{deLaValleePoussin1919,Zygmund1959}. We compare \eqref{eq:circle} with the classical delayed
de la Vall\'ee--Poussin means. Let 
\[\Proj_{\le m}f(\theta)
 :=\sum_{|k|\le m}\widehat f(k)e^{ik\theta}\]
be the Fourier
projection onto \(\mathbb P_m(\mathbb T)\), where
$\widehat f(k)
 :=\frac1{2\pi}\int_0^{2\pi}
 f(\theta)e^{-ik\theta}\,\dd\theta$.
For integers \(s\ge2\), we define the de la Vall\'ee--Poussin operator by
\[
 \mathcal H_{L,sL}f
 :=\frac{1}{(s-1)L}
   \sum_{m=L}^{sL-1}\Proj_{\le m}f,
\]
which maps \(C(\mathbb T)\) into
\(\mathbb P_{sL-1}(\mathbb T)\) and reproduces
\(\mathbb P_L(\mathbb T)\).  Its Fourier
multiplier equals one for \(|k|\le L\), decreases linearly to zero for
\(L<|k|<sL\), and vanishes for \(|k|\ge sL\).  Thus
\(\mathcal H_{L,sL}\) is one particular operator in the
admissible class defining
\(
 \lambda_1(L,sL-1)
\).

A natural comparison is provided by the specific de la Vall\'ee--Poussin
operators studied in
\cite{DeregowskaFoucartLewandowskaSkrzypek2018}.  For the 
operator \(\mathcal H_{L,sL}:C(\mathbb T)\to\mathbb P_{sL-1}\), they proved
\[
 \norm{\mathcal H_{L,sL}}-1\asymp s^{-1}.
\]
They also showed that, when \(s>2\), \(\mathcal H_{L,sL}\) is not minimal
among all bounded linear operators $\mathcal A:C(\mathbb T)\to\mathbb P_{sL-1}$
satisfying $\mathcal A|_{\mathbb P_L}=\id$
in the sense that its operator norm is strictly larger than the corresponding
generalized-projection constant.  Corollary~\ref{cor:circle} complements
this result by determining the latter to sharp order:
\[
 \lambda_1(L,sL-1)-1\asymp s^{-2}.
\]
Thus the de la Vall\'ee--Poussin operators are not only nonminimal
for \(s>2\); their excess norm is of larger asymptotic order than the
smallest possible one.

\subsection{Filtered hyperinterpolation}
\label{sec:hyper}

Hyperinterpolation discretizes the \(L^2\)-orthogonal projection by
cubature; see \cite{AnRanWu2025survey,Sloan1995}.  Its uniform operator
norm, like that of the underlying projection, grows with the polynomial
degree.  Filtered hyperinterpolation replaces the sharp spectral cutoff by
a smooth filter, retaining exact low-degree reproduction while achieving
better localization and uniform stability; see, e.g.,
\cite{Sloan2011,SloanWomersley2012,WangSloan2017}.

We apply Theorem~\ref{thm:intro} to the filtered hyperinterpolation
operators of Sloan and Womersley \cite{SloanWomersley2012}.  Let
\(h:[0,\infty)\to\mathbb R\) satisfy the regularity assumptions of
\cite{SloanWomersley2012}, with \(h(t)=1\) for \(0\le t\le1\) and
\(h(t)=0\) for \(t\ge a\), where \(a>1\), and let
\[
 N_L:=\lceil aL\rceil-1.
\]
We define the associated filtered kernel
\[
 K_L(x,y)
 :=
 \sum_{\ell=0}^{N_L}
 h\left(\frac{\ell}{L}\right)\mathcal Z_\ell(x,y).
\]
Given a positive cubature rule
\(\{(x_j,w_j)\}_{j=1}^Q\), the filtered
hyperinterpolation operator is defined by
\begin{equation}\label{eq:SW-hyper}
 \mathcal F_Lf(x)
 :=
 \sum_{j=1}^Q
 w_j K_L(x,x_j)f(x_j).
\end{equation}
If the cubature rule is exact of degree at least \(L+N_L\), then
\(\mathcal F_L\) reproduces \(\mathbb P_L\). 

The uniform boundedness theorem of Sloan and Womersley
\cite{SloanWomersley2012} states that, for each fixed such filter \(h\),
the operator norms are uniformly bounded in \(L\):
\[
 \sup_{L\ge1}
 \norm{\mathcal F_L}_{C(\Sd)\to C(\Sd)}
 \le C_{d,h}.
\]
Theorem~\ref{thm:intro}
gives a complementary lower bound that quantifies their separation from
the positivity threshold \(1\).

\begin{corollary}\label{cor:SW-lower}
Suppose that, for each \(L\ge1\), the cubature rule used in
\(\mathcal F_L\) is exact of degree at least \(L+N_L\).
For every \(L\ge1\),
\begin{equation}\label{eq:SW-lower}
 \norm{\mathcal F_L}_{C(\Sd)\to C(\Sd)}
 \ge
 1+c_d
 \left(\frac{L}{\lceil aL\rceil}\right)^2.
\end{equation}
Consequently, for every fixed \(a>1\),
\begin{equation}\label{eq:SW-liminf}
 \liminf_{L\to\infty}
 \norm{\mathcal F_L}_{C(\Sd)\to C(\Sd)}
 \ge
 1+c_da^{-2}.
\end{equation}
\end{corollary}

\begin{proof}
The operator \(\mathcal F_L\) has range in
\(\mathbb P_{N_L}\) and reproduces \(\mathbb P_L\).  Applying the lower
bound in Theorem~\ref{thm:lower} with \(N=N_L\) gives
\[
 \norm{\mathcal F_L}_{C(\Sd)\to C(\Sd)}-1
 \ge
 c_d\left(\frac{L}{N_L+1}\right)^2
 =
 c_d\left(\frac{L}{\lceil aL\rceil}\right)^2.
\]
Letting \(L\to\infty\) proves \eqref{eq:SW-liminf}.
\end{proof}

Corollary~\ref{cor:SW-lower} complements this classical upper bound by
showing that the norm must remain quantitatively separated from \(1\).
The lower bound depends only on the ratio between the reproduced and output
bandwidths, and not on the particular transition profile or the admissible
cubature rule.

\subsection{Maximum principles}
\label{sec:effective-maximum-principles}
Let \(\mathcal A:C(\Sd)\to\mathbb P_N\) reproduce \(\mathbb P_L\).
Since \(1\in\mathbb P_L\), the operator is unital, and
Proposition~\ref{prop:general-defect} identifies
\(\mathfrak p(\mathcal A)\) with both the worst undershoot below \(0\)
and the worst overshoot above \(1\) over inputs \(0\le f\le1\).
For fixed \(0<\kappa<1\) and \(1\le L\le\kappa N\),
the estimate \eqref{eq:intro-main-positivity-loss} in
Theorem~\ref{thm:intro} therefore implies that, for every admissible \(\mathcal A\) and every
\(\varepsilon>0\), there exist \(0\le f_-,f_+\le1\) and
\(x_-,x_+\in\Sd\) such that
\[
 \mathcal Af_-(x_-)
 <
 -\frac{c_d}{2}\left(\frac{L}{N+1}\right)^2+\varepsilon,
 \qquad
 \mathcal Af_+(x_+)
 >
 1+\frac{c_d}{2}\left(\frac{L}{N+1}\right)^2-\varepsilon.
\]
Thus exact low-frequency reproduction and finite bandwidth rule out an maximum principle on \([0,1]\), and the smallest possible
worst-case violation is of quadratic order.

This provides a static operator-level counterpart to effective
maximum principles for spectral discretizations of evolution
equations; see \cite{Li2021Effective} for tori and
\cite[Chapter~5]{Wu2023Thesis} for spheres.  Such results typically fix a
spatial spectral discretization, together with the underlying evolution
equation and often a time-stepping scheme, and control the numerical
solution in an enlarged invariant interval.  Here the quadratic obstruction
arises already from finite spectral bandwidth and exact reproduction,
before effects from time discretization, aliasing, or nonlinear evolution
are introduced.

{\small
\bibliographystyle{siamplain}  
\bibliography{myref}
}

\end{document}